\documentclass{amsart}

\usepackage{amsmath,amsthm,amssymb,mathtools,amscd,color,comment}
\usepackage{tikz}
\usetikzlibrary{intersections, calc, arrows.meta}
\usepackage[all]{xy}
\usepackage{enumerate}
\usepackage{url}
\usepackage{setspace}
\allowdisplaybreaks

\theoremstyle{definition}
\newtheorem{theorem}{Theorem}[section]
\newtheorem{proposition}[theorem]{Proposition}
\newtheorem{lemma}[theorem]{Lemma}
\newtheorem{corollary}[theorem]{Corollary}

\newtheorem{problem}[theorem]{Problem}

\newtheorem{definition}[theorem]{Definition}

\theoremstyle{remark}
\newtheorem{remark}[theorem]{Remark}
\numberwithin{equation}{section}

\newcommand{\T}{{\mathrm T}}
\newcommand{\R}{{\mathbb R}}
\newcommand{\Z}{{\mathbb Z}}
\newcommand{\N}{{\mathbb N}}
\newcommand{\C}{{\mathbb C}}

\newcommand{\CP}{\mathbb{CP}}
\newcommand{\D}{\Delta^{d-1}}

\newcommand{\Pol}{\mathrm{Pol}}

\newcommand{\Hom}{\mathrm{Hom}}

\newcommand{\TFF}{\mathrm{TFF}}

\newcommand{\order}{\mathcal{O}}

\title{Explicit Construction of Spherical $5$- and $7$-Designs}

\author[R.~Misawa]{Ryutaro Misawa}
\address[R.~Misawa]{Graduate School of Information Sciences\\
 Tohoku University\\
6-3-09 Aramaki-Aza-Aoba, Aoba-ku, Sendai 980-8579\\
	 Japan}
\email{misawa.ryutaro.q2@dc.tohoku.ac.jp}
\begin{document}

\begin{abstract}
This paper develops an explicit and implementable framework for constructing spherical designs by lifting point sets from tight fusion frames.
By combining existing ingredients, we obtain, in every dimension, explicit spherical $5$-designs with $|X|=\order(d^3)$.
As a core component of the method, we give an explicit construction of simplex $3$-designs realized as orbits of the symmetric group.
Using these simplex designs as input, we further construct spherical $7$-designs in arbitrary even dimensions; more precisely, for every even integer $d\ge 6$ we obtain spherical $7$-designs in dimension $d$, and if $\frac{d}{2}-1$ is a prime power then the number of points is $\order(d^6)$.

\end{abstract}

\maketitle

\section{Introduction}
Spherical $t$-designs are fundamental objects appearing in approximation theory, algebraic combinatorics, and related areas (see, e.g., \cite{ACSW2010, BB2009}).
A finite subset $X$ of the unit sphere 
\[S^{d-1}:=\{x\in \R^d \mid || x||=1\}\]
is called a (weighted) spherical $t$-design if, for every function of total degree at most $t$, the (weighted) average of its values on $X$ coincides with the integral with respect to the normalised surface measure $\sigma$.
When all weights are equal, we simply call it a spherical $t$-design; when arbitrary positive real weights are allowed, we call it a weighted spherical $t$-design.

Spherical $t$-designs were introduced by Delsarte--Goethals--Seidel \cite{DGS1977}.
Let $N(d,t)$ denote the minimum cardinality of a spherical $t$-design on $S^{d-1}$.
Using linear programming, Delsarte--Goethals--Seidel~\cite{DGS1977} obtained the following lower bound.
\begin{proposition}[{\cite{DGS1977}}]\label{prop:tight}
Fix natural numbers $d,t$. Then $N(d,t)$ satisfies
\begin{equation}\label{eq:DGS-bound}
    N(d,t)\ge
    \begin{cases}
        \binom{d+\frac{t}{2}-1}{d-1}+\binom{d+\frac{t}{2}-2}{d-1}& \text{if }t\ \text{is even},\\[2mm]
        2\binom{d+\frac{t-1}{2}-1}{d-1}& \text{if }t\ \text{is odd}.
    \end{cases}
\end{equation}
\end{proposition}
A spherical $t$-design on $S^{d-1}$ achieving equality in \eqref{eq:DGS-bound} is called a \emph{tight $t$-design}.
On $S^{1}$, it is known that a regular $(t+1)$-gon is a tight $t$-design (see, e.g., \cite{DGS1977}).
However, for $S^{d-1}$ with $d\ge 3$, tight $t$-designs are extremely rare.
Bannai--Damerell~\cite{BannaiDamerell1979,BannaiDamerell1980} showed that the only possible values of $t$ for which tight spherical $t$-designs may exist are
$t\in\{1,2,3,4,5,7,11\}$.
Moreover, classification is complete for $t=1,2,3,11$, whereas the cases $t=4,5,7$ have remained open for a long time (See \cite{BB2009}).

If $t$ is fixed and $d\to\infty$, then the bound \eqref{eq:DGS-bound} suggests that a tight $t$-design, if it exists, should satisfy
\[
|X|=\order\bigl(d^{\lfloor t/2\rfloor}\bigr).
\]
Here and throughout, for functions $f(d)$ and $g(d)$ we write $g(d)=\order(f(d))$ as $d\to\infty$ if there exist constants $C>0$ and $d_0\in\N$ such that
\[
|g(d)|\le C\,|f(d)| \qquad (\forall\, d\ge d_0).
\]
When $t$ is fixed, the constants $C$ and $d_0$ may depend on $t$.
Therefore, a natural goal is to obtain an \emph{explicit construction that is close to this order-optimal behavior}.

\begin{problem}\label{prob:construct_dual}
Fix $t \in \N$. Give an explicit construction of spherical $t$-designs $X_d\subset S^{d-1}$ such that
\[
|X_d|=\order(d^{\lfloor \frac{t}{2}\rfloor}).
\]
\end{problem}

Since the cases $t=1,2,3$ are already well understood from the viewpoint of explicit constructions, we focus on the simplest remaining cases, namely $t=4,5,7$, and summarize the current progress on Problem~\ref{prob:construct_dual}.
A point to keep in mind is that, depending on the literature, statements are often made while mixing spherical $t$-designs and \emph{weighted} spherical $t$-designs.
In particular, in numerical analysis, weighted designs are also called cubature formulas of strength $t$,
and equal-weight designs are sometimes called Chebyshev-type cubature formulas
\cite{Cools1997,Stroud1971,RB1991}.
Below, we restrict attention to families that exist in infinitely many dimensions.


\begin{enumerate}
    \item {$t=4$ and $t=5$.}
Weighted spherical designs have long been studied from the viewpoint of numerical analysis.
In particular, Stroud~\cite{Stroud1971} constructed many weighted designs for small values of $t$.
As a general framework yielding explicit constructions with fewer points in higher dimensions,
Victoir~\cite{Victoir2004} proposed a method using orthogonal arrays, producing weighted spherical $5$-designs with $\order(d^2)$ points.
For equal weights, Bajnok~\cite{Bajnok1991-1} showed that for sufficiently large $N$ one can construct spherical $5$-designs on $S^{d-1}$ with $N$ points.
Baladram~\cite{Baladram-phd} also gave an inductive construction using low-dimensional examples, still of exponential order.
Furthermore, Mohammadpour--Waldron~\cite{MW2013} showed that, assuming the Zauner conjecture, one can construct a spherical $5$-design with $6d^2$ points in even dimensions $d$.
Levenshtein~\cite{Levenshtein1982} constructed a spherical $4$-design with $d(d+1)$ points for $d=4^m$.
Kuperberg~\cite{Kuperberg2006} constructed spherical $5$-designs with $\order(d^2)$ points for $d=2^m$.

    \item{$t=7$.}
Kuperberg \cite{Kuperberg2006HatBox} gave an explicit construction of weighted spherical $7$-designs on $S^{d-1}$ for all $d$, and stated that the number of points is $\order(d^4)$.
On the other hand, compared with the cases $t=4,5$, there is much less general theory for constructing spherical $7$-designs.
Sidelnikov \cite{Sidelnikov1999} constructed examples of exponential order for $d=2^m$.
Later, Kuperberg \cite{Kuperberg2006HatBox} showed that, in dimensions $d$ for which a Hadamard matrix of order $d$ exists, one can construct spherical $7$-designs with $\order(d^6)$ points.
\end{enumerate}

In this paper, we use the following construction scheme, originating in K\"onig \cite{Koenig1999} and Kuperberg \cite{Kuperberg2006}, and further developed in Okuda \cite{Okuda2015}, Lindblad \cite{Lindblad2023}, and Misawa \cite{Misawa2026}.
For integers $d>k$, we choose finitely many $k$-dimensional subspaces $V\subset\R^d$ and place suitable designs on the unit spheres $S(V)\simeq S^{k-1}$ contained in each $V$.
By lifting these configurations to a point set on the unit sphere $S^{d-1}$ of $\R^d$, we obtain spherical designs.
As our main results, within this framework we give explicit constructions primarily of spherical $5$-designs and spherical $7$-designs.
Moreover, by focusing on the case $(d,k)=(2d',2)$ ($d'\in\N$), we combine several related analogues—simplex designs, projective toric designs, and interval designs—and present the constructions in an explicit and implementable form.

In particular, we describe the simplex design, which forms the core component of our method.
In design theory, simplex designs are analogues of spherical designs on a simplex, reformulated in a systematic manner by Baladram \cite{Baladram2018}.
For a positive integer $d$, define the standard simplex in $\R^d$ by
\[
\Delta^{d-1}
:=\Bigl\{(x_1,\dots,x_d)\in\mathbb{R}^d\ \Bigm|\ \sum_{i=1}^d x_i=1,\ \ x_i\ge 0\ (1\le i\le d)\Bigr\}.
\]
A finite subset $X\subset\Delta^{d-1}$ is called a simplex $t$-design if, for every function $f(x)=f(x_1,\dots,x_d)$ of total degree at most $t$, the average of $f$ over $X$ agrees with the integral of $f$ with respect to the normalized $(d-1)$-dimensional Lebesgue measure in $\Delta^{d-1}$. One can pose problems for simplex designs analogous to those for spherical designs, but the construction theory for simplex designs is far less developed.
Indeed, for $t=1$, the gravity point of $\Delta^{d-1}$ gives a simplex $1$-design, and for $t=2$, Baladram \cite{Baladram2018} constructed simplex $2$-designs with $d$ points for every $d$.
However, for $t\ge 3$, no explicit constructions are known in general.
In this paper, we construct simplex $3$-designs using the natural action of the symmetric group of degree $d$, denoted $S_d$.

We summarize the contributions of this paper as follows.

\subsection*{Contributions}
\begin{enumerate}
  \item Toward Problem~\ref{prob:construct_dual}, by combining known constructions we give, for $t=4$ and $t=5$ and for \textbf{every} $d$, explicit spherical $t$-designs with $|X|=\order(d^3)$.
  \item For \textbf{every} $d$, We give an explicit construction of simplex $3$-designs.
  \item For $t=7$, we provide constructions of spherical $7$-designs in arbitrary even dimensions; in particular, for $d=2q$ with $q$ a prime power, we obtain spherical $7$-designs on $S^{d-1}$ with $|X|=\order(d^6)$.

\end{enumerate}

The paper is organized as follows.
In Section~2 we introduce the designs used in our constructions.
In Section~3 we construct spherical $5$-designs.
In Section~4 we construct simplex $3$-designs as orbits of the symmetric group of degree $d$, and then use them to construct spherical $7$-designs.

\section{Preliminaries}
In this section, we introduce the notions of designs needed for our constructions and recall basic properties when necessary.
We begin with spherical designs, allowing positive weights. Let $\sigma$ be the normalized surface measure on $S^{d-1}$.

\begin{definition}\label{def:weighted-sph-design}
Let $t,d\in\N$.
A pair $(X,\lambda)$ consisting of a finite set $X=\{x_1,\dots,x_n\}\subset S^{d-1}$ and positive weights
$\lambda_1,\dots,\lambda_n>0$ is called a \emph{weighted spherical $t$-design} if for every function
$f:\R^d\to\R$ of total degree at most $t$,
\[
\int_{S^{d-1}} f(x)\,d\sigma(x)
=
\frac{1}{\Lambda}\sum_{j=1}^n \lambda_j\, f(x_j)
,
\]
holds, where $\Lambda:=\sum_{j=1}^n \lambda_j$.
\end{definition}

\begin{remark}\label{rem:weight-cases}
In Definition~\ref{def:weighted-sph-design},
\begin{enumerate}
\item[(i)]
If all weights are equal (i.e., $\lambda_1=\cdots=\lambda_n$), then $(X,\lambda)$ is a \emph{spherical $t$-design}.
\item[(ii)]
Allowing general positive real weights $\lambda_j>0$ leads to the notion of a weighted spherical $t$-design.
\end{enumerate}
\end{remark}
Below, as analogues of spherical designs, we define several related notions: interval designs, simplex designs, projective toric designs, and complex projective designs.

\subsection{Interval $t$-design}
\begin{definition}[{\cite[Sec.~1]{Nishimura2003}}]\label{def:interval_t_design}
Let $d\ge2$ and set $w_d(x):=(1-x^2)^{\frac{d-2}{2}}$ on $[-1,1]$.
A finite set $X\subset[-1,1]$ is an \emph{interval $t$-design with respect to $w_d$} if for every polynomial $f$ with
$\deg f\le t$,
\[
\frac{1}{\alpha_d}\int_{-1}^{1} f(x)\,w_d(x)\,dx
=
\frac{1}{|X|}\sum_{x\in X} f(x),
\qquad
\alpha_d:=\int_{-1}^{1} w_d(x)\,dx.
\]
\end{definition}

Using the change of variables $u^2=t$ in the definition of the beta function $B(a,b)=\int_0^1 t^{a-1}(1-t)^{b-1}\,dt$, we obtain
\[
\int_{-1}^1 u^{2m}(1-u^2)^{d-1}\,du
=
B\!\left(m+\tfrac12,d\right)
\qquad(m\in\mathbb Z_{\ge0}).
\]
Hence
\begin{align}
\frac{1}{\alpha_{2d}}\int_{-1}^1 u^{2}(1-u^2)^{d-1}\,du
&=\frac{B(\tfrac32,d)}{B(\tfrac12,d)}
=\frac{1}{2d+1},
\label{eq:interval-moment2}\\
\frac{1}{\alpha_{2d}}\int_{-1}^1 u^{4}(1-u^2)^{d-1}\,du
&=\frac{B(\tfrac52,d)}{B(\tfrac12,d)}
=\frac{3}{(2d+1)(2d+3)}.
\label{eq:interval-moment4}
\end{align}

\subsection{Simplex $t$-design}

Let $\rho$ be the normalized $(d-1)$-dimensional Lebesgue measure in $\Delta^{d-1}$.
\begin{definition}[{\cite[Eq.~(1.1)]{Baladram2018}}]\label{def:simplex_t_design}
A finite set $X\subset\Delta^{d-1}$ is called a \emph{simplex $t$-design} if
\[
\int_{\Delta^{d-1}} f(x)\,d\rho(x)
\;=\;
\frac{1}{|X|}\sum_{x\in X} f(x)
\]
holds for every polynomial $f(x)=f(x_1,\dots,x_d)$ of total degree at most $t$.
\end{definition}
The multivariate beta function is defined by
\[
B(\alpha_1,\dots,\alpha_d)
:=\frac{\prod_{i=1}^d \Gamma(\alpha_i)}{\Gamma\!\Bigl(\sum_{i=1}^d \alpha_i\Bigr)},
\qquad (\alpha_i>0),
\]
where the gamma function $\Gamma$ is given by
\[
\Gamma(\alpha):=\int_{0}^{\infty} t^{\alpha-1}e^{-t}\,dt,
\qquad (\alpha>0).
\]
By linearity, it suffices to verify Definition~\ref{def:simplex_t_design} for monomials
$x_1^{k_1}\cdots x_d^{k_d}$ with $k_1+\cdots+k_d\le t$.
Moreover, the integral representation of $B$ over the simplex
(see \cite[Eq.~(2.2)]{Baladram2018}) yields the following moment formula.

\begin{proposition}[{\cite[Eq.~(2.2) \& Eq.~(2.3)]{Baladram2018}}]
\label{prop:simplex-moment}
Let $d\in\Z_{>0}$.
Then for any nonnegative integers $k_1,\dots,k_d$,
\[
\int_{\Delta^{d-1}} x_1^{k_1}x_2^{k_2}\cdots x_d^{k_d}\,d\rho(x)
=
(d-1)!\,B(k_1+1,\dots,k_d+1).
\]
\end{proposition}

\begin{corollary}\label{cor:simplex-p2p3-moments}
For $d\in\Z_{>0}$, we have
\begin{align}
\int_{\Delta^{d-1}} p_2(x)\,d\rho(x) &= \frac{2}{d+1}, \label{eq:int-simplex-p2}\\
\int_{\Delta^{d-1}} p_3(x)\,d\rho(x) &= \frac{6}{(d+1)(d+2)}. \label{eq:int-simplex-p3}
\end{align}
\end{corollary}

\begin{proof}
Since $p_2(x)=\sum_{i=1}^d x_i^2$ and $\rho$ is symmetric under coordinate permutations,
\[
\int_{\Delta^{d-1}} p_2(x)\,d\rho(x)
= d\int_{\Delta^{d-1}} x_1^2\,d\rho(x).
\]
Applying Proposition~\ref{prop:simplex-moment} with $(k_1,k_2,\dots,k_d)=(2,0,\dots,0)$, we get
\[
\int_{\Delta^{d-1}} x_1^2\,d\rho(x)
=(d-1)!\,B(3,1,\dots,1)
=(d-1)!\,\frac{\Gamma(3)\Gamma(1)^{d-1}}{\Gamma(d+2)}
=\frac{2}{d(d+1)}.
\]
Hence \eqref{eq:int-simplex-p2} follows.

Similarly, since $p_3(x)=\sum_{i=1}^d x_i^3$,
\[
\int_{\Delta^{d-1}} p_3(x)\,d\rho(x)
= d\int_{\Delta^{d-1}} x_1^3\,d\rho(x).
\]
Applying Proposition~\ref{prop:simplex-moment} with $(k_1,k_2,\dots,k_d)=(3,0,\dots,0)$, we get
\[
\int_{\Delta^{d-1}} x_1^3\,d\rho(x)
=(d-1)!\,B(4,1,\dots,1)
=(d-1)!\,\frac{\Gamma(4)\Gamma(1)^{d-1}}{\Gamma(d+3)}
=\frac{6}{d(d+1)(d+2)}.
\]
Therefore \eqref{eq:int-simplex-p3} holds.
\end{proof}

\subsection{Projective toric designs}\label{subsec:projective_toric}

In this subsection, following \cite[Sec.~2, Def.~2.4--2.5]{IMAEG2024},
we define designs on the projective torus.
Let $d\in\N$ and consider the $d$-dimensional torus
\[
\T^d := (\R/2\pi\Z)^d.
\]
Let
\[
\T := \{\,(\theta,\dots,\theta)\in\T^d \mid \theta\in\R/2\pi\Z\,\}
\]
be the diagonal subgroup, and call the quotient
\(
P(\T^d) := \T^d/\T
\)
the \emph{projective torus}.
Let $\mu$ be the normalized Haar  measure on $P(\T^d)$.
Next, we define monomials on the projective torus.
Let $s\in\N$ and set $I_d:=\{1,\dots,d\}$.
For $a=(a_1,\dots,a_s),\,b=(b_1,\dots,b_s)\in I_d^s$ and $\varphi=(\varphi_1,\dots,\varphi_d)\in\T^d$, define
\[
m_{a,b}(\varphi)
:= \exp\!\Bigl(i\sum_{k=1}^{s}(\varphi_{a_k}-\varphi_{b_k})\Bigr).
\]
This function is invariant under diagonal translations $\varphi\mapsto\varphi+(\theta,\dots,\theta)$, and hence it is well-defined as a function on $P(\T^d)$ (denoted by the same symbol $m_{a,b}$).

We then define the following function spaces:
\[
\Pol_t\bigl(P(\T^d)\bigr)
:=\operatorname{span}\{\,m_{a,b}\mid 0\le s\le t,\ a,b\in I_d^s\,\}.
\]

\begin{definition}[{\cite[Def.~2.5]{IMAEG2024}}]
\label{def:projective_toric_design}
Let $t\in\N$.
A finite subset $X\subset P(\T^d)$ is called a \emph{projective toric $t$-design}
if for every $f\in\Pol_t\bigl(P(\T^d)\bigr)$,
\[
\int_{P(\T^d)} f(x)\,d\mu(x)
=
\frac{1}{|X|}\sum_{x\in X} f(x)
\]
holds.
\end{definition}

Moreover, Iosue et al.~\cite{IMAEG2024} explicitly constructed projective toric $t$-designs with $\order(d^t)$ points for every $d$ when $t$ is fixed:

\begin{proposition}[{\cite[Sec.~3.1]{IMAEG2024}}]\label{prop:IMAEG-poly-size}
Fix $t\in\N$.
\begin{enumerate}
\item
If $d-1$ is a prime power, then there exists a projective toric $t$-design
$X\subset P(\mathrm{T}^d)$ that is isomorphic to the cyclic group $\Z_{|X|}$, whose order is given by
\[
|X|=\frac{(d-1)^{t+1}-1}{d-2}
=1+(d-1)+\cdots+(d-1)^t.
\]
\item
For a general $d$, let $m\ge d$ be the smallest integer such that $m-1$ is a prime power.
Then there exists a projective toric $t$-design $X\subset P(\mathrm{T}^d)$ satisfying
\[
|X|=\frac{(m-1)^{t+1}-1}{m-2}.
\]
In particular, if $t$ is fixed, then $|X|=\order(d^t)$.
\end{enumerate}
\end{proposition}

\subsection{Complex projective $t$-design}
Let $d\in\N$.
Let
\[
\CP^{d-1}:=\{V\le\C^d\mid \dim_{\C} V=1\}
\]
be the complex projective space endowed with the normalized Fubini--Study measure $\mu_{\mathrm{FS}}$.
For each $k\ge0$, set
\[
\Hom_{k,k}(\C^d)
:=
\Bigl\{\,P(z,\overline z)\in
\C[z_1,\dots,z_d,\overline z_1,\dots,\overline z_d]
\ \Bigm|\ 
\deg_z P=\deg_{\overline z}P=k
\Bigr\},
\]
the space of bihomogeneous polynomials of bidegree $(k,k)$.
For $P\in\Hom_{k,k}(\C^d)$ and $[x]\in\CP^{d-1}$, choose a unit representative $x\in\C^d$ and define
\[
P([x]) := P(x,\overline x).
\]
This value is independent of the choice of a unit representative, hence well-defined.
Define
\[
\Pol_t(\CP^{d-1})
:=
\operatorname{span}\Bigl\{\,P|_{\CP^{d-1}}\ \Bigm|\ 
P\in\Hom_{k,k}(\C^d),\ 0\le k\le t
\Bigr\}.
\]

\begin{definition}[{\cite{Hoggar1982}}]
A finite set $X\subset\CP^{d-1}$ is called a \emph{complex projective $t$-design} if for every $f\in\Pol_t(\CP^{d-1})$,
\[
\int_{\CP^{d-1}} f(x)\,d\mu_{\mathrm{FS}}(x)
\;=\;
\frac1{|X|}\sum_{x\in X} f(x)
\]
holds.
\end{definition}

Moreover, Kuperberg~\cite{Kuperberg2006} and Iosue et al.~\cite{IMAEG2024} showed that a $t$-design on $\CP^{d-1}$
can be obtained from a simplex $t$-design and a projective toric $t$-design:

\begin{proposition}[{\cite[Sec.~4.1]{IMAEG2024}}]\label{prop:concatenation}
Let $Y$ be a $t$-design on the simplex $\Delta^{d-1}$, and let $X$ be a $t$-design on the projective torus $P(\mathrm{T}^d)$.
Consider the map
\[
  \pi:\ \Delta^{d-1}\times P(\mathrm{T}^d)\longrightarrow\CP^{d-1}, 
  \qquad
  (p, [\varphi])\longmapsto
  \Bigl[\, \sum_{n=1}^d \sqrt{p_n}\, e^{i\varphi_n}|n\rangle\, \Bigr].
\]
Then
\(
  Z := \pi(Y\times X)
\)
is a $t$-design on $\CP^{d-1}$.
Here $\{|n\rangle\}_{n=1}^d$ denotes the standard orthonormal basis of $\C^{d}$.
\end{proposition}

\subsection{Tight $t$-fusion frames}

Finally, we define tight $t$-fusion frames, which serve as inputs to the lifting construction, and state the corresponding lifting theorem.
Let $d,k \in \N$ with $1\le k\le d$, and let
\[
G_{k,d}:=\{\,V\le \R^d \mid \dim V = k\,\}
\]
denote the Grassmannian of $k$-dimensional subspaces of $\R^d$.

\begin{definition}\cite[Def.~4.1]{BE2013}\label{def:p-fusion-frame}
Let $\{V_j\}_{j=1}^n\subset G_{k,d}$ be a finite set and let $\{\omega_j\}_{j=1}^n$ be positive weights.
For each $j$, let $P_{V_j}:\mathbb{R}^d\to\mathbb{R}^d$ denote the orthogonal projection onto $V_j$.
If there exist constants $A,B>0$ such that for all $x\in\mathbb{R}^d$,
\[
A\,\|x\|^{2t}\ \le\ \sum_{j=1}^n \omega_j\,\|P_{V_j}x\|^{2t}\ \le\ B\,\|x\|^{2t},
\]
then $\{(V_j,\omega_j)\}_{j=1}^n$ is called a \emph{$t$-fusion frame}.
If, in particular, $A=B$, then it is called a \emph{weighted tight $t$-fusion frame}.
In the equal-weight case, we call $\{(V_j,\omega_j)\}_{j=1}^n$ a \emph{tight $t$-fusion frame} ($\TFF_t$) and often omit the weights, simply writing $\{V_j\}_{j=1}^n$.
\end{definition}

We can obtain a spherical $(2t+1)$-design on $S^{d-1}$ from a tight $t$-fusion frame on $G_{k,d}$ and a spherical $(2t+1)$-design on $S^{k-1}$.
\begin{theorem}[{\cite[Theorem~3.4]{Misawa2026}}]\label{thm:lifting-weighted}
Let $t,s\in\N$. Let $\mathcal{F}\subset G_{k,d}$ be a finite set with positive weights $\{\omega_V\}_{V\in \mathcal{F}}$, and assume that
$(\mathcal{F},\{\omega_V\})$ is a weighted tight $t$-fusion frame.
For each $V\in \mathcal{F}$, let $(Y_V,\{\lambda_{V,z}\}_{z\in Y_V})$ be a weighted spherical $s$-design on $S(V)$, and assume that
\[
\sum_{z\in Y_V}\lambda_{V,z}
\]
is independent of $V$.
Define a weighted finite set $(Z,w)$ on $S^{d-1}$ as follows:
\[
Z:=\bigcup_{V\in \mathcal{F}} Y_V \subset S^{d-1},
\qquad
w(z):=\sum_{\substack{V\in \mathcal{F}\\ z\in Y_V}}\omega_V\,\lambda_{V,z}\quad (z\in Z).
\]
Then $(Z,w)$ is a weighted spherical $r$-design on $S^{d-1}$ with
\(
r=\min\{s,\,2t+1\}.
\)
\end{theorem}
In particular, in the equal-weight case we obtain the following statement.
\begin{corollary}[{\cite[Corollary~3.5 \& Remark~3.6]{Misawa2026}}]\label{cor:lifting}
Let $t,s\in\N$.
Assume that $\mathcal{F}\subset G_{k,d}$ is a tight $t$-fusion frame, and that for each $V\in \mathcal{F}$,
$Y_V\subset S(V)$ is a spherical $s$-design, with $|Y_V|$ independent of $V$.
Then there exist $g_V\in O(V)$ for each $V\in \mathcal{F}$ such that the sets $g_VY_V$ are pairwise disjoint, and
\[
Z=\bigsqcup_{V\in \mathcal{F}} g_VY_V \subset S^{d-1}
\]
is an ordinary spherical $\min\{s,\,2t+1\}$-design.
\end{corollary}

Moreover, Bachoc--Ehler~\cite{BE2013} showed that a tight $t$-fusion frame on $G_{2,2d}$ can be obtained from a $t$-design on $\CP^{d-1}$.
\begin{proposition}[{\cite[Thm.~6.4]{BE2013}}]\label{prop:proj-to-ff}
Let $Y\subset \CP^{d-1}$ be a $t$-design.
Identify $\C^{d}$ with $\R^{2d}$ via
\[
\C^{d}\longrightarrow \R^{2d},\qquad
z=x+iy\longmapsto (x,y).
\]
Then
\[
\mathcal{F}:=\{\,\mathrm{span}_{\R}\{\,v,\,iv\,\}\mid [v]\in Y\,\}\ \subset\ G_{2,2d}
\]
is a tight $t$-fusion frame in $G_{2,2d}$.
\end{proposition}

\section{Construction of spherical $5$-designs on $S^{d-1}$}
In this section, we explain the construction method used in this paper and explicitly construct spherical $5$-designs on $S^{d-1}$ with
\(
|X|=\order(d^3).
\)

\subsection{On the construction method}

First, by Corollary~\ref{cor:lifting}, in order to obtain a $(2t+1)$-design on $S^{d-1}$, it suffices to construct:
\begin{enumerate}
  \item\label{c1} a $(2t+1)$-design on $S^{k-1}$, and
  \item\label{c2} a tight $t$-fusion frame on $G_{k,d}$.
\end{enumerate}

In this subsection we take $(d,k)=(2d',2)$ for $d'\in\N$. In this case, \eqref{c1} is provided by the vertices of a regular $(2t+2)$-gon on $S^1$.
Hence, if we can construct \eqref{c2}, namely a tight $t$-fusion frame on $G_{2,2d'}$ for arbitrary $d'$,
then we obtain a $(2t+1)$-design on $S^{2d'-1}$. In the case $(d,k)=(2d',2)$, the construction considered in this paper can be organized schematically as in Figure~\ref{fig:pipeline},
and is essentially reduced to constructing simplex $t$-designs. Here we write $(M,t)$ for a space $M$ together with the parameter $t$ specifying the relevant notion of design (or tight $t$-fusion frame) on $M$.

\begin{figure}[htbp]
  \centering
  \begin{tikzpicture}
    \draw[->] (0,0)--(2,0);
    \draw(-0.25,1/2)--(0,0);
    \draw(-0.25,-1/2)--(0,0);
    \draw[->] (4,0)--(6,1/2);
    \draw(7.76,1/2)--(8,0);
    \draw(7.76,-1/2)--(8,0);
    \draw[->] (8,0)--(10,0);

    \draw (-1.05,1/2)node{$(\Delta^{d'-1},t)$};
    \draw (-1.05,-1/2)node{$(P(\T^{d'}),t)$};
    \draw (-1.05,-4/5)node[below]{Prop.~\ref{prop:IMAEG-poly-size}};
    \draw (3,0)node{$(\CP^{d'-1},t)$};
    \draw (6.9,1/2)node{$(G_{2,2d'},t)$};
    \draw (6.9,-1/2)node{$(S^1,2t+1)$};
    \draw (6.9,-4/5)node[below]{By~\cite{DGS1977}};
    \draw (11.3,0)node{$(S^{2d'-1},2t+1)$};

    \draw (1,-1/4)node{Prop.~\ref{prop:concatenation}};
    \draw (5,-1/4)node{Prop.~\ref{prop:proj-to-ff}};
    \draw (9,-1/4)node{Cor.~\ref{cor:lifting}};
  \end{tikzpicture}
  \caption{Pipeline of the construction in the case $(d,k)=(2d',2)$.}
  \label{fig:pipeline}
\end{figure}
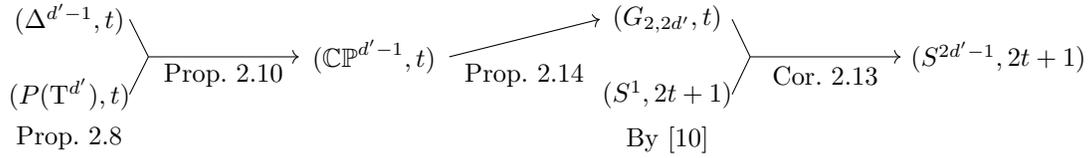

Furthermore, Rabau--Bajnok~\cite{RB1991} gave a method to construct a spherical $t$-design in one higher dimension from an interval $t$-design and a spherical $t$-design:

\begin{theorem}[{\cite[Theorem~4.1]{RB1991}}]\label{thm:rabaubajnok}
Let $t\in\N$ and $d\ge 2$.
Let $Y$ be a spherical $t$-design on $S^{d-1}\subset\R^{d}$, and let $U\subset[-1,1]$ be an interval $t$-design with respect to the weights $w_d$.
Then
\[
  X
  :=\Bigl\{
    (u, \sqrt{1-u^2}\, y)\in S^{d}
    \ \Bigm|\ 
    u\in U,\ y\in Y
  \Bigr\}
\]
is a spherical $t$-design on $S^{d}\subset\R^{d+1}$.
\end{theorem}

\begin{remark}
The Rabau--Bajnok method is usually used to construct spherical $t$-designs on $S^{d-1}$ from a spherical $t$-design on $S^{1}$ and an interval $t$-design.
However, for the following reasons, it is difficult to use this method alone to construct spherical $t$-designs with small cardinality for all $d$ with $t$ fixed:
\begin{enumerate}
  \item At each step one needs an explicit interval $t$-design with respect to the Gegenbauer measure,
  but such general explicit constructions are known only in very limited cases.
  \item Even if one can obtain an interval $t$-design at each step, the number of points increases by
  ``the number of points in the spherical $t$-design from the previous step times the number of points in the interval $t$-design''
  each time the dimension is increased by $1$.
  Hence the final number of points grows exponentially in $d$   (See \cite{Kuperberg2006,Baladram-phd}).
\end{enumerate}
\end{remark}

Applying Theorem~\ref{thm:rabaubajnok} to our construction, as long as one can construct an interval $t$-design with respect to the weights $w_d$,
one obtains a $(2t+1)$-design on $S^{2d'}$ from a $(2t+1)$-design on $S^{2d'-1}$ obtained via Figure~\ref{fig:pipeline}.

\subsection{Construction of spherical $5$-designs on $S^{d-1}$}
As shown in Figure~1 (Subsection~3.1), to construct a spherical $5$-design on $S^{2d-1}$ it suffices to construct a simplex $2$-design. Baladram~\cite{Baladram2018} has already constructed a $d$-point simplex $2$-design on $\D$.
\begin{proposition}[{\cite[Cor.~4.1]{Baladram2018}}]\label{cor:baladram-4.1}
For $d\ge 2$, let $C_d$ be the cyclic group of order $d$, acting on $\D$ by coordinate permutations.
Let
\(
x=(a,\ldots,a,\ 1-(d-1)a) \in \D
\)
with
\[
a\in\left\{\frac1d\pm \frac{1}{d\sqrt{d+1}}\right\}.
\]
Then $C_d\cdot x$ is a $d$-point simplex $2$-design on $\D$.
\end{proposition}
Therefore we obtain the following.
\begin{proposition}
\label{prop:cp-design-from-simplex-toric}
For every $d\ge 2$, a $\TFF_2$ on $G_{2,2d}$ can be constructed with $\mathcal{O}(d^{3})$ points.
\end{proposition}

\begin{proof}
Let $Y$ be a $d$-point simplex $2$-design on $\D$ by Proposition~\ref{cor:baladram-4.1}, and let $X$ be a $2$-design on $P(\mathrm{T}^d)$ by Proposition~\ref{prop:IMAEG-poly-size} with $|X|=\mathcal{O}(d^{2})$ points.
Applying Proposition~\ref{prop:concatenation} to $Y$ and $X$, we obtain a $2$-design
$Z:=\pi(Y\times X)\subset\CP^{d-1}$ with
\[
|Z|\le |Y||X|=\mathcal{O}(d^3).
\]
Finally, by Proposition~\ref{prop:proj-to-ff}, we obtain a $\TFF_2$ on $G_{2,2d}$ with at most $\mathcal{O}(d^3)$ points.
\end{proof}
Finally, applying Corollary~\ref{cor:lifting}, we obtain the following.

\begin{theorem}
    \label{thm:even_construction}
There exists a spherical $5$-design $X_d\subset S^{2d-1}$ such that
\(
|X_d|=\mathcal{O}\!\left(d^{3}\right).
\)
\end{theorem}

\begin{proof}
By Proposition~\ref{prop:cp-design-from-simplex-toric}, there exists a $\TFF_2$
\(
\mathcal F\subset G_{2,2d}
\)
with $|\mathcal F|=\mathcal O(d^3)$.
For each $V\in\mathcal F$, let $Y_V\subset S(V)\simeq S^1$ be the vertex set of a regular hexagon.
Then $Y_V$ is a spherical $5$-design on $S(V)$, and hence Corollary~\ref{cor:lifting} yields a spherical $5$-design
\[
X_d:=\bigsqcup_{V\in\mathcal F} g_V Y_V \ \subset\ S^{2d-1}
\]
for suitable $g_V\in O(V)$.
In particular, $|X_d|=6|\mathcal F|=\mathcal O(d^3)$.
\end{proof}

According to Subsection~3.1, in order to construct a spherical $5$-design in odd dimensions, it suffices to construct, for each $d\in\N$,
an interval $5$-design with respect to the weight
\(
  w_{2d}(u) \;=\; (1-u^2)^{d-1}.
\)
Equivalently, it is not necessary to construct interval $5$-designs with respect to the weight $(1-u^2)^{\frac{d-1}{2}}$;
it suffices to do so only for this specific family.

\begin{proposition}
\label{prop:Vd_interval5-design}
Let $d\in\N$ and set
\[
A_d= \frac{1}{4(2d+1)}, \quad
B_d=\frac{13\sqrt{2d+3}+\sqrt{330d-177}}{8(2d+1)\sqrt{2d+3}},\quad
C_d=\frac{13\sqrt{2d+3}-\sqrt{330d-177}}{8(2d+1)\sqrt{2d+3}}.
\]
Then
\[
  V_d
  \;:=\;
  \bigl\{\, 0,\ \pm\sqrt{A_d},\ \pm\sqrt{B_d},\ \pm\sqrt{C_d}\, \bigr\}
  \;\subset (-1, 1)
\]
is a $7$-point interval $5$-design with respect to the weight $w_{2d}$.
\end{proposition}

\begin{proof}
Since $V_d=-V_d$, we have, for every $j\in\N$,
\(
\sum_{v\in V_d} v^{2j+1}=0.
\)
Therefore, for $t=5$, it suffices to verify the even moments (degrees $2,4$):
\[
\frac17\sum_{v\in V_d}v^2= M_2(d),
\quad
\frac17\sum_{v\in V_d}v^4= M_4(d),
\]
where
\[
M_2(d):=\frac{1}{\alpha_{2d}}\int_{-1}^1 u^2 w_{2d}(u)\,du,
\qquad
M_4(d):=\frac{1}{\alpha_{2d}}\int_{-1}^1 u^4 w_{2d}(u)\,du.
\]
By \eqref{eq:interval-moment2}, \eqref{eq:interval-moment4},
\[
M_2(d)=\frac{1}{2d+1},
\qquad
M_4(d)=\frac{3}{(2d+1)(2d+3)}.
\]
Next, set
\[
T_d:=\frac{7}{2}M_2(d)-A_d=\frac{13}{4(2d+1)},
\qquad
P_d:=\frac{T_d^2-\Bigl(\frac{7}{2}M_4(d)-A_d^2\Bigr)}{2}
      =\frac{2d+171}{16(2d+1)^2(2d+3)},
\]
and then $B_d$ and $C_d$ are the roots of $z^2-T_dz+P_d=0$. We show that $V_d\subset(-1,1)$.
It is clear that $A_d\in(0,1)$. Moreover, since $T_d>0$ and $P_d>0$, we have $B_d,C_d>0$.
For $d\ge2$, we have
\[
B_d\le B_d+C_d=T_d=\frac{13}{4(2d+1)}<1,
\]
hence $0<C_d<B_d<1$. For $d=1$,
\[
B_1=\frac{13\sqrt5+\sqrt{153}}{24\sqrt5}<1,
\qquad
C_1=\frac{13\sqrt5-\sqrt{153}}{24\sqrt5}>0,
\]
so again $0<C_1<B_1<1$.
Therefore $0<A_d,B_d,C_d<1$ for all $d\in\N$, and hence $V_d\subset(-1,1)$.

By $B_d+C_d=T_d$ and $B_dC_d=P_d$, we have
\[
B_d^2+C_d^2=T_d^2-2P_d=\frac{7}{2}M_4(d)-A_d^2.
\]
Therefore,
\[
A_d+B_d+C_d=\frac{7}{2}M_2(d),\qquad
A_d^2+B_d^2+C_d^2=\frac{7}{2}M_4(d).
\]
It follows that
\[
\frac17\sum_{v\in V_d} v^2=\frac{2(A_d+B_d+C_d)}{7}=M_2(d),\qquad
\frac17\sum_{v\in V_d} v^4=\frac{2(A_d^2+B_d^2+C_d^2)}{7}=M_4(d),
\]
so $V_d$ is an interval $5$-design with respect to the weight $w_{2d}$.
\end{proof}

Hence we obtain the following.

\begin{theorem}\label{thm:odd_construction}
There exists a spherical $5$-design $X\subset S^{2d}$ such that
\(
|X|=\mathcal O(d^3).
\)
\end{theorem}

\begin{proof}
By Theorem~\ref{thm:even_construction}, there exists a spherical $5$-design
$Y\subset S^{2d-1}$ with $|Y|=\mathcal O(d^3)$.
By Proposition~\ref{prop:Vd_interval5-design}, the set $U:=V_d\subset[-1,1]$ is a $7$-point interval $5$-design with respect to the weight $w_{2d}$.
Applying Theorem~\ref{thm:rabaubajnok} to $Y$ and $U$, we obtain a spherical $5$-design
\[
X:=\Bigl\{(u,\sqrt{1-u^2}\,y)\in S^{2d}\ \Bigm|\ u\in U,\ y\in Y\Bigr\}.
\]
In particular, $|X|=|U||Y|=\mathcal O(d^3)$.
\end{proof}

Combining Theorem~\ref{thm:even_construction} and Theorem~\ref{thm:odd_construction}, we obtain the following.

\begin{theorem}
There exists a spherical $5$-design $X\subset S^{d-1}$ such that
\(
|X|=\mathcal O(d^3).
\)
\end{theorem}

\section{Construction of spherical $7$-designs on $S^{2d-1}$}
As shown in Figure~1 (Subsection~3.1), to construct a spherical $7$-design on $S^{2d-1}$ it suffices to construct a simplex $3$-design in $\Delta^{d-1}$.
In this section, we explicitly construct simplex $3$-designs as orbits of points in $\Delta^{d-1}$ under the natural action of the symmetric group $S_d$ on $\R^d$ by permuting coordinates.
Although the resulting designs are typically very large, we will later show that the number of points can be reduced when $d-1$ is a prime power.

For each non-negative integer $k$, we define the $k$-th power sum by
\[
p_k(x_1,\dots,x_d):=\sum_{i=1}^d x_i^k,
\]
and for $1\le k\le d$, the $k$-th elementary symmetric polynomial by
\[
e_k(x_1,\dots,x_d):=\sum_{1\le i_1<\cdots<i_k\le d} x_{i_1}\cdots x_{i_k}.
\]
It is well known that every symmetric polynomial in $\R[x_1,\dots,x_d]$ can be expressed as a polynomial in $e_1,\dots,e_d$.

The following lemma reduces the verification of the simplex $3$-design property for $S_d$-invariant sets to checking only $p_2$ and $p_3$.

\begin{lemma}\label{lem:sd-inv-simplex3}
Assume that a finite set $X\subset\Delta^{d-1}$ is $S_d$-invariant.
Then the following are equivalent:
\begin{enumerate}
\item $X$ is a simplex $3$-design.
\item \[
\frac1{|X|}\sum_{x\in X} p_2(x)=\frac{2}{d+1}
\quad\text{and}\quad
\frac1{|X|}\sum_{x\in X} p_3(x)=\frac{6}{(d+1)(d+2)}.
\]
\end{enumerate}
\end{lemma}

\begin{proof}
Assume (1).
Applying Definition~\ref{def:simplex_t_design} with $t=3$ to $f=p_2$ and $f=p_3$, and using
\eqref{eq:int-simplex-p2} and \eqref{eq:int-simplex-p3}, we obtain (2).
Thus it remains to prove (2)$\Rightarrow$(1).
For any polynomial $f\in\R[x_1,\dots,x_d]$ of degree $\le3$, define
\[
f^{\mathrm{sym}}(x):=\frac1{d!}\sum_{\sigma\in S_d} f(\sigma x).
\]
Since $X$ is $S_d$-invariant, we have
\[
\sum_{x\in X} f(x)=\sum_{x\in X} f^{\mathrm{sym}}(x).
\]
Moreover, the right-hand side of Proposition~\ref{prop:simplex-moment} is invariant under permutations of $(k_1,\dots,k_d)$.
Hence, for any $\sigma\in S_d$ and any polynomial $f$,
\[
\int_{\Delta^{d-1}} f(\sigma x)\,d\rho=\int_{\Delta^{d-1}} f(x)\,d\rho,
\]
and therefore
\[
\int_{\Delta^{d-1}} f\,d\rho=\int_{\Delta^{d-1}} f^{\mathrm{sym}}\,d\rho.
\]
Consequently, to prove that $X$ is a simplex $3$-design, it suffices to show that for every symmetric polynomial $g$ of degree $\le 3$,
\[
\frac1{|X|}\sum_{x\in X} g(x)=\int_{\Delta^{d-1}} g\,d\rho.
\]
If $g$ is a symmetric polynomial of degree $\le3$, then $g$ is a polynomial in $p_1$, $p_2$, $p_3$.
Since $p_1\equiv1$ in $\Delta^{d-1}$, any such $g$ can be written as
\[
g(x)=a+b\,p_2(x)+c\,p_3(x)\qquad
\text{for some }a,b,c\in\R.
\]
On the other hand, by \eqref{eq:int-simplex-p2} and \eqref{eq:int-simplex-p3}, and by the assumption (2),
the desired equality holds for $p_2$ and $p_3$, hence for all such $g$ by linearity.
Therefore $X$ is a simplex $3$-design.
\end{proof}
We now explicitly construct simplex $3$-designs as orbits $S_d\cdot x\subset\Delta^{d-1}$ under the natural coordinate-permutation action of $S_d$.
For a real number $\alpha$ and an integer $m\ge0$, we write $\alpha^{\,(m)}$ for the $m$-tuple $(\alpha,\dots,\alpha)$.
Thus, for example,
\[
(a,b^{\,(q)},c^{\,(q)})
:=
\bigl(a,\underbrace{b,\ldots,b}_{q\ \text{times}},\underbrace{c,\ldots,c}_{q\ \text{times}}\bigr)\in\R^{2q+1},
\]
and
\[
(a,b^{\,(q-1)},c^{\,(q-1)},0)
:=
\bigl(a,\underbrace{b,\ldots,b}_{q-1\ \text{times}},\underbrace{c,\ldots,c}_{q-1\ \text{times}},0\bigr)\in\R^{2q}.
\]
\begin{theorem}\label{thm:simplex3design}
Let $d\ge3$.
Define a point $x\in\Delta^{d-1}$ and the $S_d$-orbit $X_d:=S_d\cdot x$ as follows.
\begin{enumerate}
\item[(i)] If $d=2q+1$ for some $q\ge1$, define
\begin{align*}
P(s)
&=(q+1)^2(2q+1)(2q+3)\,s^3
-6(q+1)^2(2q+3)\,s^2
+3(2q+3)^2\,s
-2(2q+5).
\end{align*}
and
\[
s_-:=\frac{2(q+1)-\sqrt{2(q+1)}}{(2q+1)(q+1)}.
\]
Choose a root $s\in\bigl(s_-,\,\frac1{q+1}\bigr)$ of $P(s)=0$, and set
\[
a:=1-qs,\qquad
t:=\frac{((q+1)s-1)^2}{2(q+1)},\qquad
D:=s^2-4t,\qquad
b:=\frac{s+\sqrt{D}}2,\quad c:=\frac{s-\sqrt{D}}2.
\]
Let $x:=(a,b^{\,(q)},c^{\,(q)})$.

\item[(ii)] If $d=2q$ for some $q\ge3$, define
\begin{align*}
P(s)
&:= q(q-1)(q+1)(2q-1)(2q+1)\, s^3
 -6q(q-1)(q+1)(2q+1)\, s^2 \\
&\hspace{3.8em}
 +3(q+1)(4q^2-3)\, s
 -2(q+2)(2q-1).
\end{align*}
and
\[
s_-=\frac{2(q-1)(2q+1)\;-\;\sqrt{\,2(q-1)(2q-3)(2q+1)\,}}{(q-1)(2q-1)(2q+1)}.
\]
Choose a root $s\in\bigl(s_-,\,\frac1q\bigr)$ of $P(s)=0$, and set
\[
a:=1-(q-1)s,\qquad
t=\frac{q(q-1)(2q+1)s^2-2(q-1)(2q+1)s+(2q-1)}{2(q-1)(2q+1)},
\]
\[
D:=s^2-4t,\qquad
b:=\frac{s+\sqrt{D}}2,\quad c:=\frac{s-\sqrt{D}}2.
\]
Let $x:=(a,b^{\,(q-1)},c^{\,(q-1)},0)$.

\item[(iii)] If $d=4$, let $a\in(0,\tfrac12)$ be a root of
\[
120a^3-90a^2+18a-1=0
\]
such that the quantities $t$ and $D$ defined below satisfy $t>0$ and $D>0$.
Set $s:=1-2a$ and
\[
t:=\frac{2a^2+s^2-\frac{2}{5}}{2},\qquad
D:=s^2-4t,\qquad
b:=\frac{s+\sqrt D}{2},\quad c:=\frac{s-\sqrt D}{2},
\]
and let $x:=(a,a,b,c)$.
\end{enumerate}
Then $x\in\Delta^{d-1}$ and $X_d:=S_d\cdot x$ is a simplex $3$-design.
\end{theorem}

\begin{proof}
Throughout the proof, we use the identities
\begin{equation}\label{eq:bc-identities}
b+c=s,\quad bc=t,\quad b^2+c^2=s^2-2t,\quad b^3+c^3=s^3-3st,
\end{equation}
where $b,c$ are the roots of $z^2-sz+t=0$.
For the purpose of the proof, we regard $D=D(s)$ as a quadratic function in $s$.

\textbf{(i)}
Set
\[
s_+=\frac{2(q+1)+\sqrt{2(q+1)}}{(2q+1)(q+1)}.
\]
A direct computation shows that $D(s_+)=D(s_-)=0$.

Next,
\[
P\Bigl(\frac1{q+1}\Bigr)=\frac{2}{q+1}>0,
\]
and a direct computation gives
\[
P\bigl(s_-\bigr)
=
-\frac{(2q-1)\bigl((2q+3)\sqrt{2(q+1)}-4(q+1)\bigr)}{(q+1)(2q+1)^2}<0.
\]
Therefore, by the intermediate value theorem, there exists
\[
s\in\Bigl(s_-,\frac1{q+1}\Bigr)\ \text{with}\ P(s)=0.
\]
Moreover,
\[
s_+>\frac{2(q+1)}{(2q+1)(q+1)}
>\frac1{q+1},
\]
so $s\in\bigl(s_-,\,\frac1{q+1}\bigr)\subset (s_-,s_+)$ and hence $D(s)>0$.

Since $s>s_->0$ and $t>0$, we have $b,c>0$.
Since $s<\frac1{q+1}$,
\[
a>1-\frac{q}{q+1}=\frac1{q+1}>0,\qquad
a+q(b+c)=1,
\]
so $x=(a,b^{\,(q)},c^{\,(q)})\in\Delta^{d-1}$.

Finally, using (\ref{eq:bc-identities}), we obtain, 
\[
p_2(x)=a^2+q(b^2+c^2)=(1-qs)^2+q(s^2-2t)=\frac{2}{d+1}.
\]
Moreover,
\[
p_3(x)=a^3+q(b^3+c^3)=(1-qs)^3+q(s^3-3st),
\]
and a further computation shows that
\[
p_3(x)-\frac{6}{(d+1)(d+2)}
=
-\frac{q}{2(q+1)(2q+3)}\,P(s)=0.
\]
Thus $p_3(x)=\frac{6}{(d+1)(d+2)}$.
Now, Lemma~\ref{lem:sd-inv-simplex3} yields that $X_d$ is a simplex $3$-design.

\textbf{(ii)}
Set
\[
s_+=\frac{2(q-1)(2q+1)\;+\;\sqrt{\,2(q-1)(2q-3)(2q+1)\,}}{(q-1)(2q-1)(2q+1)}.
\]
A direct computation shows that $D(s_-)=D(s_+)=0$.

Next,
\[
P\Bigl(\frac{1}{q}\Bigr)=\frac{(2q-1)(q-1)}{q^2}>0.
\]
and a direct computation gives
\[
P(s_-)
=\frac{A(q)+\sqrt{\,2(q-1)(2q-3)(2q+1)\,}\,B(q)}{(q-1)(2q-1)^2(2q+1)}.
\]
where
\[
A(q)=2(q-1)(2q+1)(4q^2-12q+11),\qquad
B(q)=-(q+1)(2q-3)^2.
\]
For $q\ge3$ we have $B(q)<0$, so to prove $P(s_-)<0$ it suffices to show
\[
\sqrt{2(q-1)(2q-3)(2q+1)}\,(-B(q))>A(q).
\]
Squaring both sides, which are positive, reduces this to
\[
2(q-1)(2q-3)(2q+1)B(q)^2-A(q)^2
=2(q-1)(2q-1)^3(2q+1)\,F(q),
\]
where
\[
F(q)=(2q^2-6q+1)^2+q(7q-20)>0.
\]
Hence $P(s_-)<0$.
Therefore, by the intermediate value theorem, there exists
\[
s\in\Bigl(s_-,\frac1q\Bigr)\ \text{with}\ P(s)=0.
\]
Moreover,
\[
q s_+-1
=\frac{(q-1)(2q+1)\;+\;q\sqrt{\,2(q-1)(2q-3)(2q+1)\,}}{(q-1)(2q-1)(2q+1)} \;>\;0,
\]
so $s_+>\frac1q$ and hence $s\in(s_-,\frac1q)\subset(s_-,s_+)$, which implies $D(s)>0$.

Also,
\[
t=\frac{q}{2}\Bigl(s-\frac1q\Bigr)^2+\frac{1}{2q(q-1)(2q+1)}>0.
\]
Thus $b,c>0$.
Since $s<\frac1q$, we have
\[
a>\frac1q>0,
\qquad
a+(q-1)(b+c)+0=1,
\]
so $x=(a,b^{(q-1)},c^{(q-1)},0)\in\Delta^{d-1}$.

Finally, using (\ref{eq:bc-identities}), we obtain,
\[
p_2(x)=a^2+(q-1)(b^2+c^2)=\frac{2}{d+1}.
\]
Moreover, the condition $P(s)=0$ is equivalent to
\[
p_3(x)=a^3+(q-1)(b^3+c^3)=\frac{6}{(d+1)(d+2)}.
\]
Now, Lemma~\ref{lem:sd-inv-simplex3} yields that $X_d$ is a simplex $3$-design.

\textbf{(iii)}
Clearly, $2a+b+c=2a+s=1$ and $a>0$. Since $t>0$ and $D>0$, we have $b,c>0$.
Therefore $x\in\Delta^3$.
Since $X_4$ is $S_4$-invariant, by Lemma~\ref{lem:sd-inv-simplex3}, it suffices to show
$p_2(x)=\frac25$ and $p_3(x)=\frac15$.
Using (\ref{eq:bc-identities}), we obtain
\[
p_2(x)=2a^2+s^2-2t=\frac25
\]
by the definition of $t$, and
\[
p_3(x)=2a^3+s^3-3st.
\]
Substituting $s=1-2a$ and $t=\frac{2a^2+s^2-\frac25}{2}$ yields
\[
p_3(x)-\frac15=\frac1{10}\bigl(120a^3-90a^2+18a-1\bigr)=0,
\]
hence $p_3(x)=\frac15$. Therefore $X_4$ is a simplex $3$-design.
\end{proof}

By Theorem~\ref{thm:simplex3design}, we obtain an explicit construction of simplex $3$-designs in every dimension $d\ge3$.

\begin{theorem}\label{thm:sph7-existence-general}
There exists a spherical $7$-design $X_d\subset S^{2d-1}$ such that
\[
|X_d|=\mathcal{O}\!\left(2^{d}d^{9/2}\right).
\]
\end{theorem}

\begin{proof}
Let $Y\subset\Delta^{d-1}$ be the simplex $3$-design obtained in
Theorem~\ref{thm:simplex3design}.
Let $X\subset P(\mathrm T^{d})$ be a projective toric $3$-design given by
Proposition~\ref{prop:IMAEG-poly-size}\,(2), so that $|X|=\mathcal O(d^3)$.
Define $Z:=\pi(Y\times X)\subset \CP^{d-1}$ as in Proposition~\ref{prop:concatenation}$;$ then $Z$ is a complex projective $3$-design and
\[
|Z|\le |Y||X|.
\]
Applying Proposition~\ref{prop:proj-to-ff} to $Z$, we obtain a tight $3$-fusion frame
\(
\mathcal F\subset G_{2,2d}
\)
with
\[
|\mathcal F|\le |Z|\le |Y||X|.
\]
For each $V\in\mathcal F$, let $Y_V\subset S(V)\simeq S^1$ be the vertex set of a regular octagon.
Then $Y_V$ is a spherical $7$-design on $S(V)$, and hence Corollary~\ref{cor:lifting} yields a spherical $7$-design
\[
X_d:=\bigsqcup_{V\in\mathcal F} g_V Y_V \ \subset\ S^{2d-1}
\]
for suitable $g_V\in O(V)$.
In particular, $|X_d|=8|\mathcal F|$.

It remains to estimate $|Y|$.
Since $Y$ is an $S_d$-orbit, we have
\[
|Y|=
\begin{cases}
\dfrac{(2q+1)!}{(q!)^2}=(2q+1)\dbinom{2q}{q}
& \text{if } d=2q+1\text{ is odd},\\[2mm]
\dfrac{(2q)!}{((q-1)!)^2}=2q(2q-1)\dbinom{2q-2}{q-1}
& \text{if } d=2q \text{ is even}.
\end{cases}
\]
By Stirling's formula, $\binom{2n}{n}=\order(4^n/\sqrt n)$, hence regardless of the parity of $d$,
\[
|Y|=\mathcal O\!\left(2^{d}d^{3/2}\right).
\]
Consequently,
\[
|X_d|
=8|\mathcal F|
\le 8|Y||X|
=\mathcal O\!\left(2^{d}d^{9/2}\right).
\]
\end{proof}

\begin{remark}
Since the proof of Theorem~\ref{thm:sph7-existence-general} does not cover the exceptional case $d=4$, we treat this case separately.
Let $Y\subset\Delta^{3}$ be the simplex $3$-design in Theorem~\ref{thm:simplex3design}~(iii). Then $|Y|=12$.
Since $d-1=3$ is a prime power, Proposition~\ref{prop:IMAEG-poly-size}~(1) gives a projective toric $3$-design
$X\subset P(\mathrm T^{4})$ with $|X|=40$.
Hence $Z=\pi(Y\times X)\subset\CP^{3}$ satisfies $|Z|\le |Y||X|=480$, and Proposition~\ref{prop:proj-to-ff} yields a $\TFF_3$
$\mathcal F\subset G_{2,8}$ with $|\mathcal F|\le 480$.
Finally, Corollary~\ref{cor:lifting} with a regular octagon on each $V\in\mathcal F$ produces a spherical $7$-design on $S^{7}$ with
\[
|X_4|=8|\mathcal F|\le 3840.
\]
\end{remark}

\subsection{Thinning by a $3$-transitive subgroup}

Let
\[
\Omega_k
:=
\Bigl\{\, (i_1,\dots,i_k)\in I_d^k
\ \Bigm|\ 
i_r\neq i_s\ \text{for all}\ 1\le r<s\le k
\Bigr\}.
\]
Then
\[
|\Omega_k|=(d)_k:=d(d-1)\cdots(d-k+1).
\]

\begin{lemma}\label{lem:ktrans-avg}
Let $1\le k\le d$ and $G\le S_d$ be $k$-transitive on $I_d$.
Then for any function $\Phi:\Omega_k\to\mathbb R$ and any $\mathbf i\in\Omega_k$,
\[
\frac1{|G|}\sum_{g\in G}\Phi(g^{-1}\mathbf i)
=\frac1{(d)_k}\sum_{\mathbf j\in\Omega_k}\Phi(\mathbf j).
\]
\end{lemma}

\begin{proof}
Immediate.
\end{proof}

\begin{proposition}\label{prop:3trans-suborbit}
Let $d\ge 3$.
Assume that $x\in\Delta^{d-1}$ and the $S_d$-orbit $X:=S_d\cdot x$ is a simplex $3$-design.
Let $G\le S_d$ be a $3$-transitive subgroup on $I_d$ and set $Y:=G\cdot x$.
Then $Y$ is a simplex $3$-design.
\end{proposition}

\begin{proof}
It suffices to show that for every polynomial $f\in\mathbb R[y_1,\dots,y_d]$ of degree $\le 3$,
\[
\frac1{|G|}\sum_{g\in G} f(g\cdot y)
=\frac1{|S_d|}\sum_{\sigma\in S_d} f(\sigma\cdot y).
\]
By linearity, we may assume that $f=m$ is a monomial of degree $\le 3$.
Then
\[
m(y)=\prod_{t=1}^k y_{i_t}^{a_t},
\quad \text{for some } k\in I_3,\ \mathbf{i}=(i_1,\dots,i_k)\in\Omega_k,\ a_1,\dots,a_k\in\Z_{>0}
\text{ with }\sum_{t=1}^k a_t\le 3.
\]
Define $\Phi:\Omega_k\to\mathbb R$ by
\[
\Phi(j_1,\dots,j_k):=\prod_{t=1}^k x_{j_t}^{a_t}.
\]
For $g\in S_d$, since $(g\cdot x)_{i_t}=x_{g^{-1}(i_t)}$, we have $m(g\cdot x)=\Phi(g^{-1}\mathbf i)$.
As $G$ is $3$-transitive and $k\le 3$, it is $k$-transitive, hence Lemma~\ref{lem:ktrans-avg} gives
\[
\frac1{|G|}\sum_{g\in G} m(g\cdot x)
=\frac1{(d)_k}\sum_{\mathbf j\in\Omega_k}\Phi(\mathbf j).
\]
Similarly, applying Lemma~\ref{lem:ktrans-avg} to $S_d$ which is also $k$-transitive, we obtain the same right-hand side, and hence
\[
\frac1{|G|}\sum_{g\in G} m(g\cdot x)
=\frac1{|S_d|}\sum_{\sigma\in S_d} m(\sigma\cdot x).
\]
This proves the desired identity for all $f$ of degree $\le 3$, and hence $Y$ is a simplex $3$-design.
\end{proof}

\begin{corollary}\label{cor:pgl-thinning-simplex3}
Assume that $d-1=q$ is a prime power and identify $I_d$ with $\mathbb P^1(\mathbb F_q)$.
Let $x\in\Delta^{d-1}$ be such that $S_d\cdot x$ is a simplex $3$-design.
Let $G:=\mathrm{PGL}(2,q)\le S_d$ act naturally, and set $Y:=G\cdot x$.
Then $Y$ is a simplex $3$-design. In particular,  $|Y|=\mathcal{O}(d^3)$.
\end{corollary}

\begin{proof}
The natural action of $\mathrm{PGL}(2,q)$ on $\mathbb P^1(\mathbb F_q)$ is $3$-transitive.
Hence Proposition~\ref{prop:3trans-suborbit} applies and shows that $Y$ is a simplex $3$-design.
Moreover $|Y|\le |G|=q(q^2-1)=d(d-1)(d-2)$.
\end{proof}

\begin{proposition}\label{prop:tff3-size-d6}
Assume that $d\ge 3$ and that $d-1$ is a prime power.
Then one can construct a $\TFF_3$ on $G_{2,2d}$ with cardinality $\mathcal \order(d^{6})$.
\end{proposition}

\begin{proof}
For the point $x\in\Delta^{d-1}$ used in
Theorem~\ref{thm:simplex3design},
the orbit $S_d\cdot x$ is a simplex $3$-design.
By Corollary~\ref{cor:pgl-thinning-simplex3}, there exists a simplex $3$-design $Y$ such that $|Y|=\order(d^3)$. By Proposition~\ref{prop:IMAEG-poly-size}~(1), choose a projective toric $3$-design
$X\subset P(\mathrm T^d)$ with $|X|=\mathcal \order(d^3)$.
Applying Proposition~\ref{prop:concatenation} to $Y$ and $X$, we obtain a complex projective $3$-design
\(
Z:=\pi(Y\times X)\subset \CP^{d-1}
\)
with $|Z|\le |Y||X|=\mathcal \order(d^6)$.
Finally, Proposition~\ref{prop:proj-to-ff} produces a $\TFF_3$ $\mathcal{F}\subset G_{2,2d}$ with $|\mathcal{F}|\le |Z|$,
hence $|\mathcal{F}|=\mathcal \order(d^6)$.
\end{proof}

\begin{remark}\label{rem:tff3-existence}
Proposition~\ref{prop:tff3-size-d6} gives an explicit construction of a $\TFF_3$ on $G_{2,2d}$ when $d-1$ is a prime power.
For every $d\ge 3$, one can still construct a $\TFF_3$ on $G_{2,2d}$ by combining the simplex $3$-designs in
Theorem~\ref{thm:simplex3design}
with the projective toric $3$-design in Proposition~\ref{prop:IMAEG-poly-size}(2),
and then applying Propositions~\ref{prop:concatenation} and \ref{prop:proj-to-ff}.
\end{remark}

\begin{theorem}\label{thm:sph7-size-d6}
Assume that $d\ge 3$ and that $d-1$ is a prime power.
Then one can construct a spherical $7$-design $X_d\subset S^{2d-1}$ with
\[
|X_d|=\mathcal{O}\!\left(d^6\right).
\]
\end{theorem}

\begin{proof}
By Proposition~\ref{prop:tff3-size-d6}, we obtain a $\TFF_3$
$\mathcal{F}\subset G_{2,2d}$ with $|\mathcal{F}|=\mathcal{O}(d^6)$.
For each $V\in\mathcal{F}$, let $Y_V\subset S(V)\simeq S^1$ be the vertex set of a regular octagon.
By Corollary~\ref{cor:lifting}, we obtain a spherical $7$-design $X_d\subset S^{2d-1}$ with \(
|X_d|=|Y_V|\,|\mathcal{F}|=8|\mathcal{F}|=\mathcal{O}(d^6).
\)
\end{proof}

\subsection*{Acknowledgements}

The author would like to express his sincere gratitude to Professor Akihiro Munemasa for his guidance and encouragement.
He also thanks Ayodeji Lindblad for his valuable comments and helpful discussions, particularly regarding constructions of designs through projective and Hopf maps.


\end{document}